\documentclass{amsart}
\usepackage{amssymb}
\usepackage{amsfonts}

\newtheorem{theorem}{Theorem}[section]
\theoremstyle{plain}

\newtheorem{definition}[theorem]{Definition}
\newtheorem{example}[theorem]{Example}

\newtheorem{proposition}[theorem]{Proposition}
\newtheorem{remark}[theorem]{Remark}

\numberwithin{equation}{section}
\input{tcilatex}

\begin{document}
\title[Quasi-shape theory]{Quasi-shape theory of locally finite and
paracompact spaces}
\author{Andrei V. Prasolov}
\address{Institute of Maths \& Stats, Univ. of Troms\o , Norway}
\email{andrei.prasolov@uit.no}
\urladdr{http://www.math.uit.no/users/andreip/Welcome.html}
\date{December 28, 2010}
\subjclass[2000]{Primary 55P55}
\keywords{Shape, quasi-shape, locally finite spaces, weak homotopy type,
hypercoverings, pro-sets, pro-spaces.}

\begin{abstract}
Shape theory works nice for (Hausdorff) paracompact spaces, but for spaces
with no separation axioms, it seems to be quite poor. However, for finite
and locally finite spaces their weak homotopy type is rather rich, and is
equivalent to the weak homotopy type of finite and locally finite polynedra,
respectively. In the paper there is proposed a variant of shape theory
called quasi-shape, which suits both paracompact and locally finite spaces,
i.e. the quas-shape is isomorphic to the weak homotopy type for locally
finite spaces, and is $\natural $-equivalent to the ordinary shape in the
case of paracompact spaces.
\end{abstract}

\maketitle

\section{Main construction}

\subsection{The connected component functor $\protect\pi $}

We need an appropriate definition of%
\begin{equation*}
\pi :TOP\longrightarrow SETS
\end{equation*}%
where $TOP$\ and $SETS$ are the categories of topological spaces and sets,
respectively. Neither the usual functor $\pi _{0}$ (the set of pathwise
connected components) nor $\pi _{0}^{\prime }$ (the set of connected
components) is suitable for our purposes. We will introduce instead the
following functor%
\begin{equation*}
\pi :TOP\longrightarrow pro\text{-}SETS:
\end{equation*}%
\begin{equation*}
\pi \left( X\right) _{\mathcal{U}}:=\mathcal{U}
\end{equation*}%
for any \textbf{open} partition of $X$ (i.e., a partition into open
subsets). We say that $\mathcal{U}\leq \mathcal{V}$ if $\mathcal{V}$ refines 
$\mathcal{U}$. The set $Part\left( X\right) $ of all open partitions of $X$
is clearly directed, and we obtain an inverse system of sets by defining%
\begin{equation*}
p_{\mathcal{U}\leq \mathcal{V}}:\pi \left( X\right) _{\mathcal{V}%
}\longrightarrow \pi \left( X\right) _{\mathcal{U}}
\end{equation*}%
where%
\begin{eqnarray*}
p_{\mathcal{U}\leq \mathcal{V}}\left( Y\right) &=&Y^{\prime }, \\
Y &\in &\mathcal{V}, \\
Y^{\prime } &\in &\mathcal{U},
\end{eqnarray*}%
and $Y^{\prime }$ is the unique element of $\mathcal{U}$, containing $Y$.

Let now 
\begin{equation*}
f:X\longrightarrow Y
\end{equation*}%
be a continuous mapping. Define%
\begin{equation*}
\pi \left( f\right) :\pi \left( X\right) \longrightarrow \pi \left( Y\right)
\end{equation*}%
by the following. Let $\mathcal{U}$ be an open partition of $Y$, and let 
\begin{equation*}
\mathcal{V}=\xi \left( \mathcal{U}\right) =\xi _{f}\left( \mathcal{U}\right)
:=\left\{ Y^{\prime }=f^{-1}\left( Y\right) :Y\in \mathcal{U},Y^{\prime
}\neq \varnothing \right\} .
\end{equation*}%
$\mathcal{V}$ is clearly an open partition of $X$, and we have just defined
a mapping%
\begin{equation*}
\xi =\xi _{f}:Part\left( Y\right) \longrightarrow Part\left( X\right) .
\end{equation*}%
There can be defined also a mapping%
\begin{equation*}
f_{\mathcal{U}}:\pi \left( X\right) _{\mathcal{V}}\longrightarrow \pi \left(
Y\right) _{\mathcal{U}}
\end{equation*}%
by%
\begin{equation*}
f_{\mathcal{U}}\left( Y^{\prime }\right) :=Y
\end{equation*}%
where%
\begin{equation*}
\varnothing \neq Y^{\prime }=f^{-1}\left( Y\right) .
\end{equation*}%
It can be easily checked that the pair%
\begin{equation*}
\left( \xi _{f}:Part\left( Y\right) \longrightarrow Part\left( X\right)
,\left( f_{\mathcal{U}}:\pi \left( X\right) _{\xi _{f}\left( \mathcal{V}%
\right) }\longrightarrow \pi \left( Y\right) _{\mathcal{U}}:\mathcal{U}\in
Part\left( Y\right) \right) \right)
\end{equation*}%
gives a well-defined morphism%
\begin{equation*}
\pi \left( f\right) :\pi \left( X\right) \longrightarrow \pi \left( Y\right)
\end{equation*}%
in the category $pro$-$SETS$, and the correspondence $f\longmapsto \pi
\left( f\right) $ defines a functor%
\begin{equation*}
\pi :TOP\longrightarrow pro\text{-}SETS.
\end{equation*}

\begin{proposition}
Let $X$ be a locally connected space. Then $\pi \left( X\right) $ is
isomorphic in the category $pro$-$SETS$ to the set $\pi _{0}^{\prime }\left(
X\right) $ of connected components of $X$.
\end{proposition}

\begin{proof}
The set $\pi _{0}^{\prime }\left( X\right) $ is an open partition of $X$
which refines any other open partition. Therefore, $Part\left( X\right) $
has a maximal element $\pi _{0}^{\prime }\left( X\right) $, and $\pi \left(
X\right) $ is isomorphic to the trivial pro-set $\pi _{0}^{\prime }\left(
X\right) $ indexed by a one-point index set, i.e. to the \textbf{set} $\pi
_{0}^{\prime }\left( X\right) $.
\end{proof}

\subsection{Quasi-shape}

Let $Cov\left( X\right) $ be the set of open coverings on $X$, pre-ordered
by the refinement relation. Analogously to $Part\left( X\right) $, $%
Cov\left( X\right) $ is a directed \textbf{pre-ordered} set, while $%
Part\left( X\right) $ is a directed \textbf{ordered} set. Let 
\begin{equation*}
U_{\cdot }=\left( U_{\ast },d_{\ast },s_{\ast }\right)
\end{equation*}%
be a hypercovering on $X$ (see \cite{Atrin-Mazur-MR883959}, Definition 8.4),
i.e. a simplicial space with an augmentation%
\begin{equation*}
\varepsilon :U_{\cdot }\longrightarrow X,
\end{equation*}%
and the following properties:

\begin{description}
\item[Hyper$_{0}$] 
\begin{equation*}
\varepsilon _{0}:U_{0}\longrightarrow X
\end{equation*}%
is an open covering;

\item[Hyper$_{n}$] 
\begin{equation*}
U_{n+1}\longrightarrow \left( Cosk_{n}U_{\cdot }\right) _{n+1}
\end{equation*}%
are open coverings, $n\geq 0$.
\end{description}

If $\mathcal{U}$ is an open covering, one can define the corresponding \v{C}%
ech hypercovering by%
\begin{equation*}
U_{n}=\dcoprod\limits_{U_{i}\in \mathcal{U}}\left( U_{0}\cap U_{1}\cap
...\cap U_{n}\right)
\end{equation*}%
with the evident face ($d_{\ast }$) and degeneracy ($s_{\ast }$) mappings,
where $\amalg $ is the coproduct in the category of topological spaces. For
the \v{C}ech hypercovering, the mappings%
\begin{equation*}
U_{n+1}\longrightarrow \left( Cosk_{n}U_{\cdot }\right) _{n+1},n\geq 0,
\end{equation*}%
are homeomorphisms.

\begin{remark}
The \v{C}ech hypercoverings are used in the definition of ordinary shape of
a topological space, see \cite{Mardesic-MR1740831}.
\end{remark}

\begin{definition}
\label{Shape}Let $X$ be a topological space. The \textbf{shape} of $X$ is
the following pro-space. Given a \textbf{normal} (i.e. admitting a partition
of unity) covering $\mathcal{U}$, let $N\mathcal{U}$ (the \v{C}ech nerve of $%
\mathcal{U}$) be a simplicial set with%
\begin{equation*}
\left( N\mathcal{U}\right) _{n}=\left\{ \left( U_{0},U_{1},...,U_{n}\right)
:\left( U_{i}\in \mathcal{U}\right) \mathbf{\&}\left( U_{0}\cap U_{1}\cap
...\cap U_{n}\neq \varnothing \right) \right\}
\end{equation*}%
with the evident face ($d_{\ast }$) and degeneracy ($s_{\ast }$) mappings.
If $\mathcal{V}$ refines $\mathcal{U}$, there exists a unique (up to
homotopy) mapping 
\begin{equation*}
p_{\mathcal{U}\leq \mathcal{V}}:N\mathcal{V}\longrightarrow N\mathcal{U}.
\end{equation*}%
The correspondence%
\begin{equation*}
\mathcal{U}\longmapsto \left\vert N\mathcal{U}\right\vert
\end{equation*}%
where $\mathcal{U}$ runs over all normal coverings on $X$, and $\left\vert N%
\mathcal{U}\right\vert $ is the geometric realization of $N\mathcal{U}$,
defines an object $SH\left( X\right) $ in $pro$-$H\left( TOP\right) $ which
is called the shape of $X$.
\end{definition}

Let $HCov\left( X\right) $ be the following category: the objects are
hypercoverings on $X$, and the morphisms from $U_{\cdot }$ to $V_{\cdot }$
are homotopy classes of simplicial mappings%
\begin{equation*}
U_{\cdot }\longrightarrow V_{\cdot }
\end{equation*}%
This category is co-filtering. Given a hypercovering $U_{\cdot }$, let%
\begin{equation*}
\Gamma \left( U_{\cdot },\pi \right) =\left\vert \pi \left( U_{\cdot
}\right) \right\vert
\end{equation*}%
where $\left\vert \pi \left( U_{\cdot }\right) \right\vert $ is the
geometric realization of the simplicial pro-set $\pi \left( U_{\cdot
}\right) $. Varying $U_{\cdot }$, one gets an object%
\begin{equation*}
U_{\cdot }\longmapsto \left\vert \pi \left( U_{\cdot }\right) \right\vert
\end{equation*}%
in $pro$-$H\left( pro\text{-}TOP\right) $. Finally, applying the canonical
functor%
\begin{equation*}
pro\text{-}H\left( pro\text{-}TOP\right) \longrightarrow pro\text{-}\left(
pro\text{-}H\left( TOP\right) \right) \longrightarrow pro\text{-}H\left(
TOP\right) ,
\end{equation*}%
one gets an object $QSH\left( X\right) $ in $pro$-$H\left( TOP\right) $
which will be called the \textbf{quasi-shape} of $X$.

\begin{theorem}
\label{QSH}The correspondence above gives a well-defined functor%
\begin{equation*}
QSH:TOP\longrightarrow pro\text{-}H\left( TOP\right) ,
\end{equation*}%
which factors through the homotopy category $H\left( TOP\right) $:%
\begin{equation*}
QSH:TOP\longrightarrow H\left( TOP\right) \longrightarrow pro\text{-}H\left(
TOP\right) .
\end{equation*}
\end{theorem}

\begin{remark}
The functor from $H\left( TOP\right) $ to $pro$-$H\left( TOP\right) $ will
be denoted $QSH$ as well
\end{remark}

\section{Comparison}

Let $X$ be a locally finite space (see \cite{McCord-MR0196744}, p. 466). It
means that every point has a finite neighborhood. Due to \cite%
{McCord-MR0196744}, Theorem 2, there exists a simplicial set $\mathcal{K}%
\left( X\right) $, functorially dependent on $X$, and a weak homotopy
equivalence%
\begin{equation*}
\left\vert \mathcal{K}\left( X\right) \right\vert \longrightarrow X.
\end{equation*}%
Let us consider the functor above as a functor to $pro$-$H\left( TOP\right) $%
:%
\begin{equation*}
X\longmapsto \left\vert \mathcal{K}\left( X\right) \right\vert :LF\text{-}%
TOP\longrightarrow TOP\subseteq pro\text{-}H\left( TOP\right)
\end{equation*}%
where $LF$-$TOP$ is the full subcategory of locally finite spaces.

\begin{example}
\label{4-Circle}Let $X$ be a so called $\mathbf{4}$\textbf{-point circle},
i.e. a space with four points $\left\{ a,b,c,d\right\} $ and the following
topology%
\begin{equation*}
\tau =\left\{ X,\varnothing ,\left\{ a\right\} ,\left\{ c\right\} ,\left\{
a,b,c\right\} ,\left\{ a,d,c\right\} ,\left\{ a,c\right\} \right\} .
\end{equation*}%
Then $\left\vert \mathcal{K}\left( X\right) \right\vert $ is homeomorphic to
an ordinary circle $S^{1}$.
\end{example}

\begin{theorem}
\label{Comparison-locally-finite}On the category%
\begin{equation*}
LF\text{-}TOP\subseteq TOP,
\end{equation*}%
there exists a natural isomorphism%
\begin{equation*}
\left\vert \mathcal{K}\left( X\right) \right\vert \approx QSH\left( X\right)
.
\end{equation*}
\end{theorem}

\begin{remark}
The shape of a locally finite (even a finite) space differs significantly
from $\left\vert \mathcal{K}\left( X\right) \right\vert $. Say, the space
from Example \ref{4-Circle} has the shape of a point.
\end{remark}

Let now $X$ be a Hausdorff paracompact space. We will simply call such
spaces \textbf{paracompact}. Remind that a $\natural $-equivalence between
pro-spaces is a mapping%
\begin{equation*}
f:\mathbf{X}\longrightarrow \mathbf{Y}
\end{equation*}%
in $pro$-$H\left( TOP\right) $ inducing an isomorphism of pro-sets%
\begin{equation*}
\pi _{0}\left( f\right) :\pi _{0}\left( \mathbf{X}\right) \longrightarrow
\pi _{0}\left( \mathbf{Y}\right) ,
\end{equation*}%
and isomorphisms of pro-groups%
\begin{equation*}
\pi _{n}\left( f\right) :\pi _{n}\left( \mathbf{X},f^{-1}\left( y\right)
\right) \longrightarrow \pi _{n}\left( \mathbf{Y},y\right) ,n\geq 1,
\end{equation*}%
for any point $y\longrightarrow \mathbf{Y}$. It is known \cite%
{Atrin-Mazur-MR883959} that the canonical morphism%
\begin{equation*}
X_{\alpha }\longrightarrow \left( Cosk_{n}X_{\alpha }\right)
\end{equation*}%
is a $\natural $-equivalence between pro-spaces%
\begin{equation*}
\mathbf{X}\longrightarrow Cosk\left( \mathbf{X}\right) .
\end{equation*}

\begin{theorem}
\label{Comparison-paracompact}Let $X$ be\ a paracompact space. Then $%
QSH\left( X\right) $ is naturally $\natural $-equivalent to the ordinary
shape $SH\left( X\right) $ of $X$.
\end{theorem}

\section{Proofs}

\subsection{Proof of Theorem \protect\ref{QSH}}

\begin{proof}
The crucial step is the following. Given two homotopic mappings%
\begin{equation*}
f,g:X\rightrightarrows Y,
\end{equation*}%
the corresponding morphisms:%
\begin{equation*}
QSH\left( f\right) =QSH\left( g\right) :QSH\left( X\right) \longrightarrow
QSH\left( Y\right)
\end{equation*}%
are equal in the category $pro$-$H\left( TOP\right) $. This, in turn is
proved using compactness of the unit interval and the technique of
Proposition (8.11) from \cite{Atrin-Mazur-MR883959}: given a hypercovering $%
U_{\cdot }$ on $Y$, one constructs a sequence of truncated hypercoverings on 
$X$, resulting in a hypercovering $V_{\cdot }$ on $X$, which refines both $%
f^{-1}\left( U_{\cdot }\right) $ and $g^{-1}\left( U_{\cdot }\right) $, and
such that the corresponding morphisms%
\begin{equation*}
\Gamma \left( V_{\cdot },\pi \right) \longrightarrow \Gamma \left( U_{\cdot
},\pi \right)
\end{equation*}%
are equal in the category $pro$-$H\left( TOP\right) $.
\end{proof}

\subsection{Proof of Theorem \protect\ref{Comparison-locally-finite}}

\begin{proof}
Introduce the following pre-order on $X$ (see \cite{McCord-MR0196744}, p.
468):%
\begin{equation*}
x\leq y\Longleftrightarrow V_{y}\subseteq V_{x}
\end{equation*}%
where $V_{x}$ is the minimal (finite) open neighborhood of $x$. Let now $%
U_{\cdot }$ be the following hypercovering:%
\begin{equation*}
U_{n}=\dcoprod\limits_{x_{0}\leq x_{1}\leq ...\leq x_{n}}V_{x_{n}}
\end{equation*}%
with the evident face and degeneracy mappings. This hypercovering is clearly
an initial object in the category $HCov\left( X\right) $. All spaces $V_{x}$
are connected, therefore, for each $n$, $\pi \left( U_{n}\right) $ is a 
\textbf{set} (i.e. a trivial pro-set). Finally, $QSH\left( X\right) $ is a 
\textbf{space} (i.e. a trivial pro-space) $\left\vert K\left( X\right)
\right\vert $ where $K\left( X\right) $ is the following simplicial set:%
\begin{equation*}
K\left( X\right) _{n}=\left\{ x_{0}\leq x_{1}\leq ...\leq x_{n}\right\} .
\end{equation*}%
The latter simplicial set is exactly the simplicial set $\mathcal{K}\left(
X\right) $ from \cite{McCord-MR0196744}, Theorem 2. It follows that%
\begin{equation*}
QSH\left( X\right) \approx \left\vert \mathcal{K}\left( X\right) \right\vert
\end{equation*}%
(homotopy equivalent) while%
\begin{equation*}
\left\vert \mathcal{K}\left( X\right) \right\vert \overset{weak}{\approx }X
\end{equation*}%
(weak homotopy equivalent).
\end{proof}

\subsection{Proof of Theorem \protect\ref{Comparison-paracompact}}

\begin{proof}
There exists \cite{Atrin-Mazur-MR883959} a natural $\natural $-equivalence%
\begin{equation*}
QSH\left( X\right) \longrightarrow Cosk\left( QSH\left( X\right) \right) .
\end{equation*}

Let now construct a homotopy equivalence%
\begin{equation*}
Sh\left( X\right) \longrightarrow Cosk\left( QSH\left( X\right) \right) .
\end{equation*}%
Let $U_{\cdot }\in HCov\left( X\right) $, let $n\in \mathbb{N}$ and let%
\begin{equation*}
V_{\cdot }=Cosk_{n}\left( QSH\left( X\right) \right) =Cosk_{n}\left( \Gamma
\left( U_{\cdot },\pi \right) \right) .
\end{equation*}%
Consider the following open covering $\mathcal{U}$ on $X$:%
\begin{equation*}
\mathcal{U}=\left( d_{0}\right) ^{n}:U_{n}\longrightarrow X.
\end{equation*}%
Let us now consider the open partitions $\mathcal{W}_{0}$, $\mathcal{W}_{1}$%
, ... , $\mathcal{W}_{n}$, of $U_{0}$, $U_{1}$, ... , $U_{n}$, involved in
the construction of pro-sets $\pi \left( U_{0}\right) $, $\pi \left(
U_{1}\right) $, ... , $\pi \left( U_{n}\right) $. Finally, since $X$ is
paracompact, there exists a normal open covering $\mathcal{V}$ on $X$,
refining $\mathcal{U}$ and all coverings%
\begin{equation*}
\left( d_{0}\right) ^{i}\mathcal{W}_{i},i=0,1,...,n.
\end{equation*}%
Denote the correspondence%
\begin{equation*}
\left( U_{.},n,\mathcal{W}_{i}\right) \longmapsto \mathcal{V}
\end{equation*}%
by%
\begin{equation*}
\xi \left( U_{.},n,\mathcal{W}_{i}\right) =\mathcal{V}.
\end{equation*}%
Given $V\in \mathcal{V}$, there exist unique elements $W_{i}$ from $\mathcal{%
W}_{i}$ such that%
\begin{equation*}
V\subseteq \left( d_{0}\right) ^{i}W_{i}.
\end{equation*}%
This gives a well-defined mapping from the \v{C}ech nerve%
\begin{equation*}
\varphi _{\left( U_{.},n,\mathcal{W}_{i}\right) }:N\mathcal{V}%
\longrightarrow Cosk_{n}\left( \Gamma \left( U_{\cdot },\pi \right) \right) .
\end{equation*}%
Finally, the pair $\left( \xi ,\varphi \right) $ gives the desired
equivalence%
\begin{equation*}
SH\left( X\right) \longrightarrow Cosk\left( QSH\left( X\right) \right)
\end{equation*}%
in $pro$-$H\left( TOP\right) $.
\end{proof}

\bibliographystyle{amsalpha}
\bibliography{QuasiShape}

\end{document}